\documentclass[12pt]{article}

\usepackage{amsmath,amsthm}
\usepackage{amssymb}
\usepackage{color}
\usepackage{xspace}
\definecolor{gray}{rgb}{0.7, 0.75, 0.71}

\usepackage[colorlinks=true,
linkcolor=red,
filecolor=brown,
citecolor=blue]{hyperref}
\def\red{\textcolor{red} }
\def\blue{\textcolor{blue} }
\def\green{\textcolor{green} }

\def\v{\vert}
\def\a{\ensuremath{\mathcal A}\xspace}
\def\b{\ensuremath{\mathcal B}\xspace}
\def\d{\ensuremath{\mathcal D}\xspace}
\def\r{\ensuremath{\mathcal R}\xspace}

\def\u{\ensuremath{\mathcal U}\xspace}
\def\V{\ensuremath{\mathcal V}\xspace}
\def\gf{generating function\xspace}
\def\gl{ground level\xspace}
\catcode`\@=11

\thicklines
\newskip\Einheit \Einheit=.6cm
\newcount\xcoord \newcount\ycoord
\newdimen\xdim \newdimen\ydim \newdimen\PfadD@cke \newdimen\Pfadd@cke
\def\PfadDicke#1{\PfadD@cke#1 \divide\PfadD@cke by2 
\Pfadd@cke\PfadD@cke \multiply\PfadD@cke by2}
\long\def\LOOP#1\REPEAT{\def\BODY{#1}\ITERATE}
\def\ITERATE{\BODY \let\next\ITERATE \else\let\next\relax\fi \next}
\let\REPEAT=\fi
\def\Punkt{\hbox{\raise-2pt\hbox to0pt{\hss\scriptsize$\bullet$\hss}}}

\def\DuennPunkt(#1,#2){\unskip
  \raise#2 \Einheit\hbox to0pt{\hskip#1 \Einheit
          \raise-1.5pt\hbox to0pt{\hss\tiny$\bullet$\hss}\hss}}
		  
\def\NormalPunkt(#1,#2){\unskip
  \raise#2 \Einheit\hbox to0pt{\hskip#1 \Einheit
          \raise-3pt\hbox to0pt{\hss\large$\bullet$\hss}\hss}}
\def\DickPunkt(#1,#2){\unskip
  \raise#2 \Einheit\hbox to0pt{\hskip#1 \Einheit
          \raise-4pt\hbox to0pt{\hss\Large$\bullet$\hss}\hss}}
\def\Kreis(#1,#2){\unskip
  \raise#2 \Einheit\hbox to0pt{\hskip#1 \Einheit
          \raise-4pt\hbox to0pt{\hss\Large$\circ$\hss}\hss}}
\def\Diagonale(#1,#2)#3{\unskip\leavevmode
  \xcoord#1\relax \ycoord#2\relax
      \raise\ycoord \Einheit\hbox to0pt{\hskip\xcoord \Einheit
         \unitlength\Einheit
         \line(1,1){#3}\hss}}
\def\AntiDiagonale(#1,#2)#3{\unskip\leavevmode
  \xcoord#1\relax \ycoord#2\relax \advance\xcoord by -0.05\relax
      \raise\ycoord \Einheit\hbox to0pt{\hskip\xcoord \Einheit
         \unitlength\Einheit
         \line(1,-1){#3}\hss}}
\def\Pfad(#1,#2),#3\endPfad{\unskip\leavevmode
  \xcoord#1 \ycoord#2 \thicklines\ZeichnePfad#3\endPfad\thinlines}
\def\ZeichnePfad#1{\ifx#1\endPfad\let\next\relax
  \else\let\next\ZeichnePfad
    \ifnum#1=1
      \raise\ycoord \Einheit\hbox to0pt{\hskip\xcoord \Einheit
         \vrule height\Pfadd@cke width1 \Einheit depth\Pfadd@cke\hss}%
      \advance\xcoord by 1
     \else\ifnum#1=2
      \raise\ycoord \Einheit\hbox to0pt{\hskip\xcoord \Einheit
         \unitlength\Einheit
         \line(0,1){1}\hss}
      \advance\xcoord by 0
      \advance\ycoord by 1
 \else\ifnum#1=3
      \raise\ycoord \Einheit\hbox to0pt{\hskip\xcoord \Einheit
         \unitlength\Einheit
         \line(1,1){1}\hss}
      \advance\xcoord by 1
      \advance\ycoord by 1
    \else\ifnum#1=4
      \raise\ycoord \Einheit\hbox to0pt{\hskip\xcoord \Einheit
         \unitlength\Einheit
         \line(1,-1){1}\hss}
      \advance\xcoord by 1
      \advance\ycoord by -1
   \else\ifnum#1=5
      \raise\ycoord \Einheit\hbox to0pt{\hskip\xcoord \Einheit
         \unitlength\Einheit
         \line(2,1){2}\hss}
      \advance\xcoord by 2
      \advance\ycoord by 1
	  \else\ifnum#1=6
      \raise\ycoord \Einheit\hbox to0pt{\hskip\xcoord \Einheit
         \unitlength\Einheit
         \line(2,-1){2}\hss}
      \advance\xcoord by 2
      \advance\ycoord by -1
	  \else\ifnum#1=7
      \raise\ycoord \Einheit\hbox to0pt{\hskip\xcoord \Einheit
         \unitlength\Einheit
         \line(3,1){3}\hss}
      \advance\xcoord by 3
      \advance\ycoord by 1
	  \else\ifnum#1=8
      \raise\ycoord \Einheit\hbox to0pt{\hskip\xcoord \Einheit
         \unitlength\Einheit
         \line(3,-1){3}\hss}
      \advance\xcoord by 3
      \advance\ycoord by -1
    \fi\fi\fi\fi\fi\fi\fi\fi
  \fi\next}
\def\hSSchritt{\leavevmode\raise-.4pt\hbox 
to0pt{\hss.\hss}\hskip.2\Einheit
  \raise-.4pt\hbox to0pt{\hss.\hss}\hskip.2\Einheit
  \raise-.4pt\hbox to0pt{\hss.\hss}\hskip.2\Einheit
  \raise-.4pt\hbox to0pt{\hss.\hss}\hskip.2\Einheit
  \raise-.4pt\hbox to0pt{\hss.\hss}\hskip.2\Einheit}
\def\vSSchritt{\vbox{\baselineskip.2\Einheit\lineskiplimit0pt
\hbox{.}\hbox{.}\hbox{.}\hbox{.}\hbox{.}}}
\def\DSSchritt{\leavevmode\raise-.4pt\hbox to0pt{%
  \hbox to0pt{\hss.\hss}\hskip.2\Einheit
  \raise.2\Einheit\hbox to0pt{\hss.\hss}\hskip.2\Einheit
  \raise.4\Einheit\hbox to0pt{\hss.\hss}\hskip.2\Einheit
  \raise.6\Einheit\hbox to0pt{\hss.\hss}\hskip.2\Einheit
  \raise.8\Einheit\hbox to0pt{\hss.\hss}\hss}}
\def\dSSchritt{\leavevmode\raise-.4pt\hbox to0pt{%
  \hbox to0pt{\hss.\hss}\hskip.2\Einheit
  \raise-.2\Einheit\hbox to0pt{\hss.\hss}\hskip.2\Einheit
  \raise-.4\Einheit\hbox to0pt{\hss.\hss}\hskip.2\Einheit
  \raise-.6\Einheit\hbox to0pt{\hss.\hss}\hskip.2\Einheit
  \raise-.8\Einheit\hbox to0pt{\hss.\hss}\hss}}
\def\SPfad(#1,#2),#3\endSPfad{\unskip\leavevmode
  \xcoord#1 \ycoord#2 \ZeichneSPfad#3\endSPfad}
\def\ZeichneSPfad#1{\ifx#1\endSPfad\let\next\relax
  \else\let\next\ZeichneSPfad
    \ifnum#1=1
      \raise\ycoord \Einheit\hbox to0pt{\hskip\xcoord \Einheit
         \hSSchritt\hss}%
      \advance\xcoord by 1
    \else\ifnum#1=2
      \raise\ycoord \Einheit\hbox to0pt{\hskip\xcoord \Einheit
        \hbox{\hskip-2pt \vSSchritt}\hss}%
      \advance\ycoord by 1
    \else\ifnum#1=3
      \raise\ycoord \Einheit\hbox to0pt{\hskip\xcoord \Einheit
         \DSSchritt\hss}
      \advance\xcoord by 1
      \advance\ycoord by 1
    \else\ifnum#1=4
      \raise\ycoord \Einheit\hbox to0pt{\hskip\xcoord \Einheit
         \dSSchritt\hss}
      \advance\xcoord by 1
      \advance\ycoord by -1
    \fi\fi\fi\fi
  \fi\next}
\def\Koordinatenachsen(#1,#2){\unskip
 \hbox to0pt{\hskip-.5pt\vrule height#2 \Einheit width.5pt depth1 
\Einheit}%
 \hbox to0pt{\hskip-1 \Einheit \xcoord#1 \advance\xcoord by1
    \vrule height0.25pt width\xcoord \Einheit depth0.25pt\hss}}
\def\Koordinatenachsen(#1,#2)(#3,#4){\unskip
 \hbox to0pt{\hskip-.5pt \ycoord-#4 \advance\ycoord by1
    \vrule height#2 \Einheit width.5pt depth\ycoord \Einheit}%
 \hbox to0pt{\hskip-1 \Einheit \hskip#3\Einheit 
    \xcoord#1 \advance\xcoord by1 \advance\xcoord by-#3 
    \vrule height0.25pt width\xcoord \Einheit depth0.25pt\hss}}
\def\Gitter(#1,#2){\unskip \xcoord0 \ycoord0 \leavevmode
  \LOOP\ifnum\ycoord<#2
    \loop\ifnum\xcoord<#1
      \raise\ycoord \Einheit\hbox to0pt{\hskip\xcoord 
\Einheit\Punkt\hss}%
      \advance\xcoord by1
    \repeat
    \xcoord0
    \advance\ycoord by1
  \REPEAT}
\def\Gitter(#1,#2)(#3,#4){\unskip \xcoord#3 \ycoord#4 \leavevmode
  \LOOP\ifnum\ycoord<#2
    \loop\ifnum\xcoord<#1
      \raise\ycoord \Einheit\hbox to0pt{\hskip\xcoord 
\Einheit\Punkt\hss}%
      \advance\xcoord by1
    \repeat
    \xcoord#3
    \advance\ycoord by1
  \REPEAT}
\def\Label#1#2(#3,#4){\unskip \xdim#3 \Einheit \ydim#4 \Einheit
  \def\lo{\advance\xdim by-.5 \Einheit \advance\ydim by.5 \Einheit}%
  \def\llo{\advance\xdim by-.25cm \advance\ydim by.5 \Einheit}%
  \def\loo{\advance\xdim by-.5 \Einheit \advance\ydim by.25cm}%
  \def\o{\advance\ydim by.25cm}%
  \def\ro{\advance\xdim by.5 \Einheit \advance\ydim by.5 \Einheit}%
  \def\rro{\advance\xdim by.25cm \advance\ydim by.5 \Einheit}%
  \def\roo{\advance\xdim by.5 \Einheit \advance\ydim by.25cm}%
  \def\l{\advance\xdim by-.30cm}%
  \def\r{\advance\xdim by.30cm}%
  \def\lu{\advance\xdim by-.5 \Einheit \advance\ydim by-.6 \Einheit}%
  \def\llu{\advance\xdim by-.25cm \advance\ydim by-.6 \Einheit}%
  \def\luu{\advance\xdim by-.5 \Einheit \advance\ydim by-.30cm}%
  \def\u{\advance\ydim by-.30cm}%
  \def\ru{\advance\xdim by.5 \Einheit \advance\ydim by-.6 \Einheit}%
  \def\rru{\advance\xdim by.25cm \advance\ydim by-.6 \Einheit}%
  \def\ruu{\advance\xdim by.5 \Einheit \advance\ydim by-.30cm}%
  #1\raise\ydim\hbox to0pt{\hskip\xdim
     \vbox to0pt{\vss\hbox to0pt{\hss$#2$\hss}\vss}\hss}%
}
\catcode`\@=12

\begin{document}
\newtheorem{theorem}{Theorem}
\newtheorem{defn}[theorem]{Definition}
\newtheorem{lemma}[theorem]{Lemma}
\newtheorem{prop}[theorem]{Proposition}
\newtheorem{cor}[theorem]{Corollary}
\begin{center}

{\Large
On Ascent, Repetition and Descent Sequences     \\ 
}

\vspace{5mm}
David Callan  \\
\today
\end{center}

\begin{abstract}
Ascent sequences have received a lot of attention in recent years in 
connection with (2 + 2)-free posets and other combinatorial objects. 
Here, we  first show bijectively that analogous repetition 
sequences  are counted by the Bell numbers, and 021-avoiding repetition sequences by the Catalan numbers.  
Then we adapt a bijection of Chen et al and use it along with the ``symbolic'' 
method of Flajolet to find the 4-variable generating function for 021-avoiding ascent sequences
by length, number of 0's, number of isolated 0's, and number of runs of 2 or more 0's. 
We deduce that 021-avoiding ascent sequences that have no consecutive 0's (resp. no isolated 0's) both 
satisfy a Catalan-like recurrence, differing only in initial conditions, and give a bijective proof for 
the case of no consecutive 0's. Lastly, we show that 021-avoiding descent sequences are equinumerous with 
same-size $UUD\hspace{2 pt}U$-avoiding Dyck paths.

\end{abstract}

\section{Introduction}\label{intro}

An \emph{ascent sequence} is a sequence $a_1a_2\dots a_n$ of nonnegative 
integers with $a_1=0$ and $a_i \le 1\:+ $ number of ascents in $a_1\dots  
a_{i-1}$ for $i\ge 2$, that is, 
$a_i \le 1 +\,\#\,\{j\in[\,1,i-2\,]:\ a_j<a_{j+1}\}$.
Analogously, \emph{repetition sequences}  and \emph{descent sequences} are defined by replacing ``$a_j<a_{j+1}$'' with ``$a_j=a_{j+1}$'' and ``$a_j>a_{j+1}$,'' respectively, in the definition of ascent sequence.
Ascent, repetition, and descent sequences are counted, respectively, by the Fishburn numbers 
\htmladdnormallink{A022493}{http://oeis.org/A022493} in 
the OEIS \cite{oeis}, the Bell numbers 
\htmladdnormallink{A000110}{http://oeis.org/A000110}, and \htmladdnormallink{A225588}{http://oeis.org/A225588}. 

Ascent sequences have received attention in recent years in connection 
with (2 + 2)-free posets and other combinatorial objects, e.g., \cite{dukes2019}. 
For avoidance of a pattern up to length 4 in ascent sequences, see \cite{dun2011}.
We have the following useful little lemma \cite{dun2011}.
\begin{lemma}\label{021crit}
For an ascent, repetition, or descent sequence, since the first entry is 0, 
avoidance of the pattern 021 (aka pattern 132) is equivalent to ``nonzero entries are weakly increasing.''
\end{lemma}
We let $\a_n$ denote the set of ascent sequences of length $n$ and $\a_n(021)$ those that avoid 021. 
Analogously, $\mathcal{R}_n$ and $\mathcal{R}_n(021)$ refer to repetition sequences, and $\d_n$ and $\d_n(021)$ to descent sequences. 
Thus, $\mathcal{R}_1=\{0\},\ \mathcal{R}_2=\{00,01\},\ \mathcal{R}_3=\{000,\, 001,\,002,\,010,\,011\}$ and 012 is 
not included in $\mathcal{R}_3$ because the last entry, 2, is too large.

In Section \ref{rsp}, we give a bijective proof that $\v \mathcal{R}_n \v =B_n$, the Bell number.
In Section \ref{021rs}, we show bijectively that $\v \mathcal{R}_n(021) \v =C_n$, the Catalan number.
In Section \ref{bij}, we give a bijection, based on a result in \cite{021chen}, from 021-avoiding ascent 
sequences to Dyck paths, and use it in Section \ref{gf021} to find the 
4-variable generating function $F(x,y,z,w)$ for 021-avoiding ascent sequences with 
$x,y,z,w$ marking, respectively, length, \# 0's, \# isolated 0's, \# runs of 2 or more 0's.
In the two sections after that,  we count 021-avoiding ascent sequences that have no consecutive 0's 
(resp. no isolated 0's) and show that, curiously, the counting sequences both satisfy a Catalan-like recurrence.
In the last section, we count 021-avoiding descent sequences and 
show that they are equinumerous with same-size $UUD\hspace{2 pt}U$-avoiding Dyck paths.

\section{From repetition sequences to partitions} \label{rsp}
We will recursively define a bijection $\phi$ from $\mathcal{R}_n$, the repetition sequences of length $n$, to set partitions of $[n]$, counted by the Bell numbers 
\htmladdnormallink{A005843}{http://oeis.org/A005843},
that sends ``number of repetitions''  to ``number of dividers.'' 
We write all set partitions in a canonical form: increasing entries within each block, and blocks arranged in increasing order of smallest entries, for example 135/29/4/678 with 4 blocks and 3 dividers (slashes) separating the blocks.

First, for $n=1$, $\phi(0)=1$ with no repetitions and no dividers. Then, for $w=a_1 a_2\dots a_n \in \mathcal{R}_n$ with $n\ge 2$, we may suppose by induction that $\phi(a_1 a_2\dots a_{n-1})$ is a set partition of $[n-1]$ in canonical form with $k$ dividers, hence $k+1$ blocks, where $k$ is the  number of repetitions in $a_1 a_2\dots a_{n-1}$. 

Now place $n$ in a block determined as follows:
\begin{itemize}
\vspace*{-3mm}
\item if $a_n=a_{n-1}$, place $n$ in a singleton block at the end, 
\item if $a_n>a_{n-1}$, place $n$ in the $a_n$-th block,
\item if $a_n<a_{n-1}$, place $n$ in the $(1+a_n)$-th block.
\end{itemize}
For example, given that $\phi(002)=1/23$ (by induction), $\phi$ sends $0020,\,0021,\,0022$ respectively to $14/23,\,1/234,\,1/23/4$. It is fairly easy to see that this procedure will work to produce a set partition with the claimed number of dividers.
It is also easy to turn $\phi$ into an explicit bijection by starting with the appropriate number of empty blocks and then placing $n,n-1,\dots,2$ in turn into their blocks and, lastly, placing 1 in the first block.

\section{021-Avoiding repetition sequences} \label{021rs}
Recall that a repetition sequence $a$ avoids 021 if and only if the 
nonzero entries of $a$ are weakly increasing left to right. 
Thus 00111020225 is a 021-avoiding repetition sequence. So
$\r_3(021)=\r_3$ and the only entry of $\r_4$ not in $\r_4(021)$ is 0021.
To show that $\v \r_n(021) \v =C_n$, we will define recursively a bijection 
$\psi$ from $\r_n(021)$ to Dyck paths of size $n$ (where size means semilength = number of up steps) 
that sends  repetitions to valleys (a valley is an occurrence of $DU$, $D$ a down step, $U$ an upstep). 
First, $\psi(0)=UD$ with no repetitions and no valleys. 

Now suppose for given $n$, $\psi(a)$ has been defined for $a \in \r_n(021)$ (induction hypothesis). 
Each element of $\r_{n+1}(021)$ is formed by 
appending a suitable entry $a_{n+1}$ to $a=(a_i)_{i=1}^n \in \r_n(021)$. We will 
show how to define $\psi$ in each case
by inserting $UD$ appropriately into $\psi(a)$. By way of illustration, let $n=9$ and
$a=000223303$ and $\psi(a)=P$ as in Figure 1. 

\Einheit=0.7cm
\[
\hspace*{5mm}
\Label\l{ \textrm{{\footnotesize 0}}}(-9.6,0.6)
\Label\l{ \textrm{{\footnotesize 1}}}(-8.6,0.6)
\Label\l{ \textrm{{\footnotesize 2}}}(-7.6,0.6)
\Label\l{ \textrm{{\footnotesize 3}}}(-6.6,0.6)
\Label\u{ \textrm{{\footnotesize 4}}}(-6,1)
\Label\u{ \textrm{{\footnotesize 5}}}(-5,1)
\Label\u{ \textrm{{\footnotesize 6}}}(-4,1)
\Label\u{ \textrm{{\footnotesize 7}}}(-3,1)
\Label\u{ \textrm{{\footnotesize 8}}}(-2,1)
\Label\u{ \textrm{{\footnotesize 9}}}(-1,1)
\Label\u{ \textrm{{\footnotesize 10}}}(0,1)
\Label\u{ \textrm{{\footnotesize 11}}}(1,1)
\Label\u{ \textrm{{\footnotesize 12}}}(2,1)
\Label\u{ \textrm{{\footnotesize 13}}}(3,1)
\Label\u{ \textrm{{\footnotesize 14}}}(4,1)
\Label\u{ \textrm{{\footnotesize 15}}}(5,1)
\Label\u{ \textrm{{\footnotesize 16}}}(6,1)
\Label\u{ \textrm{{\footnotesize 17}}}(7,1)
\Label\u{ \textrm{{\footnotesize 18}}}(8,1)
\SPfad(-10,1),111111111111111111\endSPfad
\Pfad(-10,1),3433\endPfad
\green{\Pfad(-6,3),433\endPfad}
\Pfad(-3,4),33444344434\endPfad
\DuennPunkt(-10,1)
\DuennPunkt(-9,1)
\DuennPunkt(-8,1)
\DuennPunkt(-7,1)
\DuennPunkt(-6,1)
\DuennPunkt(-5,1)
\DuennPunkt(-4,1)
\DuennPunkt(-3,1)
\DuennPunkt(-2,1)
\DuennPunkt(-1,1) 
\DuennPunkt(0,1) 
\DuennPunkt(1,1)
\DuennPunkt(2,1)
\DuennPunkt(3,1)
\DuennPunkt(4,1)
\DuennPunkt(5,1)
\DuennPunkt(6,1)
\DuennPunkt(7,1)
\DuennPunkt(8,1)
\Label\u{\swarrow}(0.5,6.0)
\Label\o{ \textrm{\small  repetition vertex}}(2.9,5.7)
\Label\u{\searrow}(-5,3.7)
\Label\o{ \textrm{\small  last $DUU$}}(-5.8,3.8)
\DuennPunkt(-9,2)
\DuennPunkt(-7,2)
\DuennPunkt(-6,3)
 \green{\DuennPunkt(-5,2)
\DuennPunkt(-4,3)}
\DuennPunkt(-3,4)
\DuennPunkt(-2,5)
\blue{\NormalPunkt(-1,6) }
\red{\NormalPunkt(0,5)} 
\DuennPunkt(1,4)
\DuennPunkt(2,3)
\blue{\NormalPunkt(3,4)}
\DuennPunkt(4,3)
\DuennPunkt(5,2)
\DuennPunkt(6,1)
\blue{\NormalPunkt(7,2)}
\DuennPunkt(8,1)
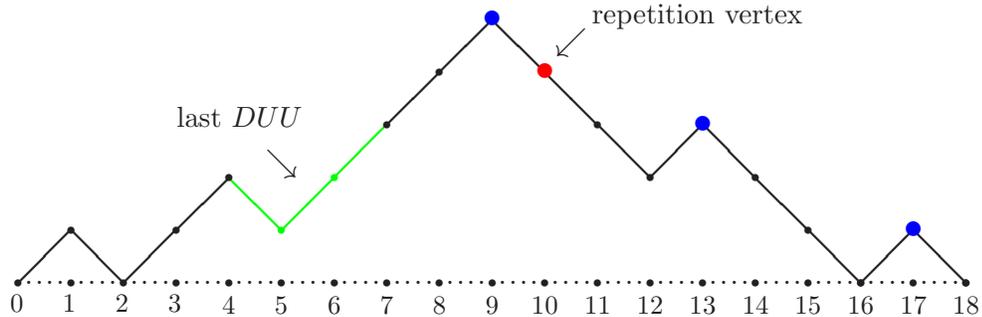
\Label\o{ \textrm{\small  The Dyck path $P$ with last $DUU$ in green, repetition vertex}}(0,-1)
\Label\o{ \textrm{\small  in red and nonrep vertices in blue}}(0,-1.8)
\Label\o{ \textrm{Figure 1}}(0,-3.2)
\]
  
\vspace*{3mm}

The valid values for $a_{n+1}$ are $a_n$, here 3, the \emph{repetition} 
value because it increments  by one the number of repetitions in the sequence, and (since $a$ 
has 4 repetitions) 0,4,5, the \emph{nonrep} values.
Their counterparts in $\psi(a)$ are defined as follows. The \emph{key peak} in a nonempty 
Dyck path $P$ is the first peak after the last $DUU$ in $P$, and the first peak 
in case $P$ has no $DUU$. 
The repetition vertex is the vertex immediately after the key peak and
the nonrep vertices are the key peak vertex and all later peak vertices.
Thus, in Figure 1, the key vertex is at location 9, the repetition vertex is at location 10 
and the nonrep vertices are at locations 9,13,17. 
The ``extreme'' cases of a pyramid path and a sawtooth path, both of which avoid $DUU$, are illustrated in Figure 2.

\Einheit=0.7cm
\[
\SPfad(-8,0),111111\endSPfad
\SPfad(2,0),111111\endSPfad
\Pfad(-8,0),333444\endPfad
\Pfad(2,0),343434\endPfad
\DuennPunkt(-8,0)
\DuennPunkt(-7,1)
\DuennPunkt(-6,2)
\DuennPunkt(-5,3)
\DuennPunkt(-4,2)
\DuennPunkt(-3,1)
\DuennPunkt(-2,0)
\DuennPunkt(2,0)
\DuennPunkt(3,1)
\DuennPunkt(4,0)
\DuennPunkt(5,1)
\DuennPunkt(6,0)
\DuennPunkt(7,1)
\DuennPunkt(8,0)
\Label\u{\swarrow}(-3.3,3)
\Label\o{ \textrm{\small  repetition vertex}}(-2,2.8)
\blue{\NormalPunkt(-5,3) }
\red{\NormalPunkt(-4,2)} 
\blue{\NormalPunkt(3,1)}
\blue{\NormalPunkt(5,1)}
\blue{\NormalPunkt(7,1)}
\red{\NormalPunkt(4,0)} 
\Label\o{ \textrm{\small  pyramid path}}(-5,-1)
\Label\o{ \textrm{\small  sawtooth path}}(5,-1)
\Label\o{ \textrm{Figure 2}}(0,-2.5)
\]
  
\vspace*{3mm}

By induction (see below) the number 
of nonrep values for $a_{n+1}$ is the same as the number of nonrep 
vertices in $\psi(a)$. The definition of $\psi$ on $\r_{n+1}(021)$ is now to simply 
insert $UD$ at the 
corresponding nonrep vertex or at the repetition vertex as appropriate. For 
example, if $a_{n+1}=5$, the third nonrep value, insert $UD$ at the peak 
at location 17, the third nonrep vertex in Figure 1, and if  $a_{n+1}=3$, insert $UD$ 
at the red vertex.

Setting $a_{n+1}$ to the repetition value increments by 1 both the number of 
repetitions and the number of valid nonrep values. Setting $a_{n+1}$ to 
the $i$-th nonrep value preserves the number of repetitions and, due to the
weakly increasing requirement on 
nonzero entries, reduces the number of valid nonrep values by $i-1$. 
Correspondingly, inserting $UD$ at the repetition vertex preserves 
the last $DUU$  and increments by 1 both the number of valleys and the number of 
nonrep vertices, while inserting $UD$ at the $i$-th nonrep vertex produces a new last $DUU$, 
kills $i-1$ of the nonrep vertices and preserves the number of valleys. 
These observations are the basis for the induction claims above. 

It is not hard to see that an all-0 sequence goes to a sawtooth path and an alternating $010\dots$ 
sequence goes to a pyramid path.

As for reversing the map, if the last two entries of $a$ are equal, the 
insertion of $UD$ ensures that the last $DUU$ (which is unchanged) starts an ascent 
that is immediately followed by a short descent (i.e., of length 1). Otherwise, the 
last $DUU$ starts an ascent that is immediately followed by a long descent, 
distinguishing the two cases, and the inverse procedure is clear.  

\section{A bijection from $\mathbf{\a(021)}$ to Dyck paths} \label{bij}
Here, based on a decomposition of $\a_{n}(021)$ due to Chen et al \cite{021chen}, we describe a bijection $\tau$ that sends $\a_n(021)$ to the Dyck paths of semilength $n$. Actually, we give two descriptions, a recursive one and an algorithmic on: recursive is more concise but algorithmic is more illuminating, showing how the image Dyck path is built up by successive insertions of a $U$ and a $D$, always at \gl, according to the successive entries of the 021-avoiding ascent sequence. The algorithmic description will be useful in the next Section.

Following \cite{021chen}, for $a=(a_i)_{i=1}^n \in \a_n$, say $i$ is a \emph{tight} index if
$a_i = 1 +\,\#\,\{j\in[\,1,i-2\,]:\ a_j<a_{j+1}\}$ so that $a_i$ has the maximum value allowed 
by  the defining restriction of an ascent sequence. We  have an almost obvious lemma.
\begin{lemma}
For $a \in \a_n(021)$, if $i$ is a tight index, then $a_i$ is an ascent top.
\end{lemma}
\begin{proof}
Suppose $i$ is tight. 
Then $a_i \ne 0$ and $a_1\dots a_i$ ends with $a_r<a_{r+1}=a_{r+2}= \cdots =a_i$ 
for some $r\le i-1$ due to the 
weakly increasing property of the nonzero entries. We wish to show $r=i-1$. If not, 
$a_{r+1}=a_i=1 +\# \textrm{ ascents in }a_1 \dots a_{r+1}\textrm{ (since $r+1 < i$) } = 2+ \# 
\textrm{ ascents in }a_1 \dots a_{r}$, and $a_{r+1}$ is too big for an ascent sequence.
\end{proof}
The \emph{key} index $k$ for $a\in \a_n(021)$ is its largest tight 
index. The index $i$ of the first 1 in $a$ is tight and so $k$ exists except for 
the all-0 sequence $0^n$, where we take $k=n$ as the key index. 
Henceforth, suppose $a\in\a_n(021)$. 
Set $M=a_k$. Then $a_{k+1}$, if present, is $M$ or 0. More generally, deleting the first $k$ entries and all $t\ge 0$ $M$s that immediately follow $a_k$, the remaining sequence $a_{k+t+1} \dots a_n$ is either empty or begins with 0 and, after each nonzero entry is 
decremented by $M-1$, is a 021-avoiding ascent sequence.  This fact is the key to recursion.

First, $\tau$ sends the empty sequence to the empty path. Now, to define $\tau$ recursively, 
suppose given $a\in \a_n(021)$ with $n\ge 1$. With $k$ the key index, if $a_{k+1}=M$, define 
\[
\tau(a)=UD\,\tau(a_1,\dots,\widehat{a_k},\dots,a_n),
\] 
where the hat denotes that entry is omitted.
Otherwise, define
\[
\tau(a)=U\,\tau(a_1,\dots,a_{k-1})\,D\,\tau(b_{k+1},\dots,b_n),
\]
where $b_{k+1},\dots,b_n$ is a 021-avoiding ascent sequence obtained from $a_{k+1},\dots,a_n$ by subtracting $M-1$ from each nonzero entry.

For the algorithmic description, we need the notion of the \emph{DD-components} of a nonempty Dyck path: split the path after each $DD$ that returns the path to \gl, see Figure 3 below.

\Einheit=0.4cm
\[
\SPfad(-12,0),111111111111111111111111\endSPfad
\Pfad(-12,0),343433434434343433344434\endPfad
\DuennPunkt(-12,0)
\DuennPunkt(-11,1)
\DuennPunkt(-10,0)
\DuennPunkt(-9,1)
\DuennPunkt(-8,0)
\DuennPunkt(-7,1)
\DuennPunkt(-6,2)
\DuennPunkt(-5,1) 
\DuennPunkt(-4,2) 
\DuennPunkt(-3,1)
\blue{\NormalPunkt(-2,0) }
\DuennPunkt(-1,1)
\DuennPunkt(-0,0)
\DuennPunkt(1,1)
\DuennPunkt(2,0)
\DuennPunkt(3,1)
\DuennPunkt(4,0)
\DuennPunkt(5,1)
\DuennPunkt(6,2)
\DuennPunkt(7,3)
\DuennPunkt(8,2)
\DuennPunkt(9,1)
\blue{\NormalPunkt(10,0)}
\DuennPunkt(11,1)
\DuennPunkt(12,0)
\Label\u{ \textrm{{\footnotesize $\uparrow$}}}(7,0)
\Label\u{ \textrm{{\footnotesize \gl}}}(7,-1)
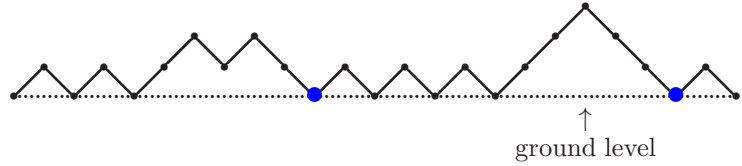
\Label\u{ \textrm{\small  A Dyck path with 3 $DD$-components, delimited by the blue vertices}}(0,-2.4)
\Label\o{ \textrm{Figure 3}}(0,-5.8)
\]  
\vspace*{3mm}

Thus each $DD$-component has the form $(UD)^i UPD$ where $i\ge 0$ and $P$ is a nonempty Dyck path, except for the last one where $P$ may be empty; in other words, the last $DD$-component may also have the form $(UD)^i,\ i\ge 1$.

Now, to obtain $\tau(a_1 \dots a_{n-1}a_n)$ from $P=\tau(a_1\dots a_{n-1})$ for $(a_i)_{i=1}^n\in\a_n(021)$, consider 
cases. If $a_n=0$, append $UD$ to $P$. If $a_n=a_{n-1}>0$, insert $UD$ just before the last $DD$-component of $P$. 
It is convenient to call all other valid values of $a_n$ the \emph{main values} of $a_n$. 
They constitute an interval of one or more integers as in the following Table, where
$m$ denotes $\max(a_1\dots a_{n-1})$ and \# asc denotes the number of ascents in $a_1\dots  a_{n-1}$.
\[
\begin{array}{cc}
\textrm{Values of } m \textrm{ and }a_{n-1}\ & \textrm{ Main values for }a_n\\ \hline
m=0 & 1 \\
m>0\textrm{ and } a_{n-1}=0\ \ & [\,m,\,1+\, \textrm{\# asc}\,] \\
m=a_{n-1}>0 & [\,m+1,\,1+\,\textrm{\# asc}\,]
\end{array}
\]
Say $a_n$ is the $j$th main value (from smallest to largest). 
Then elevate the last $j$ $DD$-components of $P$. This means that if 
$P=QP_j \dots P_2P_1$ where $P_j,\dots,P_2,P_1$ are the last $j$
$DD$-components of $P$, then $\tau(a)=Q UP_j \dots P_2P_1D$. For example, with $n=6$ and $(a_i)_{i=1}^{n-1}=01011$, $\tau(01011)$ is shown in Figure 4, and the construction of $\tau\big((a_i)_{i=1}^n\big)$ for each $a_n$ is shown in Figure 5. 

\Einheit=0.4cm
\[
\SPfad(-5,0),1111111111\endSPfad
\Pfad(-5,0),3344343344\endPfad
\DuennPunkt(-5,0)
\DuennPunkt(-4,1)
\DuennPunkt(-3,2)
\DuennPunkt(-2,1)
\DuennPunkt(-1,0)
\DuennPunkt(0,1)
\DuennPunkt(1,0)
\DuennPunkt(2,1)
\DuennPunkt(3,2)
\DuennPunkt(4,1)
\DuennPunkt(5,0)
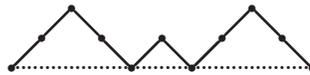
\Label\u{ \textrm{\small  The Dyck path $P=\tau\big((a_i)_{i=1}^{n-1}\big)=\tau(01011)$}}(0,-.8)
\Label\u{ \textrm{Figure 4}}(0,-3)
\]
  
\vspace*{10mm}

\Einheit=0.4cm
\[
\SPfad(-14,0),111111111111\endSPfad
\SPfad(2,0),111111111111\endSPfad
\Pfad(-14,0),334434334434\endPfad
\Pfad(2,0),334434343344\endPfad
\DuennPunkt(-14,0)
\DuennPunkt(-13,1)
\DuennPunkt(-12,2)
\DuennPunkt(-11,1)
\DuennPunkt(-10,0)
\DuennPunkt(-9,1)
\DuennPunkt(-8,0)
\DuennPunkt(-7,1)
\DuennPunkt(-6,2)
\DuennPunkt(-5,1)
\DuennPunkt(-4,0)
\DuennPunkt(-3,1)
\DuennPunkt(-2,0)
\DuennPunkt(2,0)
\DuennPunkt(3,1)
\DuennPunkt(4,2)
\DuennPunkt(5,1)
\DuennPunkt(6,0)
\DuennPunkt(7,1)
\DuennPunkt(8,0)
\DuennPunkt(9,1)
\DuennPunkt(10,0)
\DuennPunkt(11,1)
\DuennPunkt(12,2)
\DuennPunkt(13,1)
\DuennPunkt(14,0)
\Label\u{ \textrm{\small  $a_n=0$}}(-8,-.5)
\Label\u{ \textrm{\small  $a_n=1$}}(8,-.5)
\]

\[
\SPfad(-14,0),111111111111\endSPfad
\SPfad(2,0),111111111111\endSPfad
\Pfad(-14,0),334433433444\endPfad
\Pfad(2,0),333443433444\endPfad
\DuennPunkt(-14,0)
\DuennPunkt(-13,1)
\DuennPunkt(-12,2)
\DuennPunkt(-11,1)
\DuennPunkt(-10,0)
\DuennPunkt(-9,1)
\DuennPunkt(-8,2)
\DuennPunkt(-7,1)
\DuennPunkt(-6,2)
\DuennPunkt(-5,3)
\DuennPunkt(-4,2)
\DuennPunkt(-3,1)
\DuennPunkt(-2,0)
\DuennPunkt(2,0)
\DuennPunkt(3,1)
\DuennPunkt(4,2)
\DuennPunkt(5,3)
\DuennPunkt(6,2)
\DuennPunkt(7,1)
\DuennPunkt(8,2)
\DuennPunkt(9,1)
\DuennPunkt(10,2)
\DuennPunkt(11,3)
\DuennPunkt(12,2)
\DuennPunkt(13,1)
\DuennPunkt(14,0)
\Label\u{ \textrm{\small  $a_n=2$}}(-8,-.5)
\Label\u{ \textrm{\small  $a_n=3$}}(8,-.5)
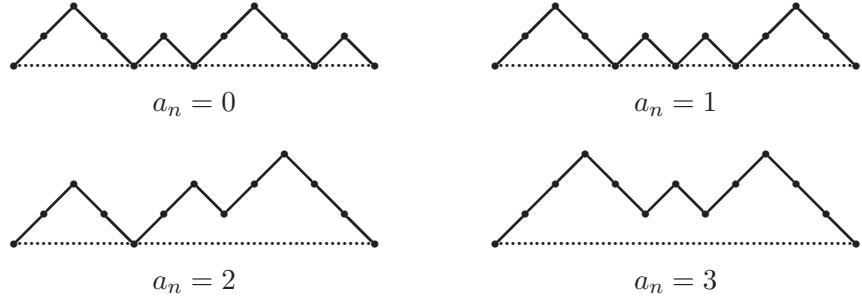
\Label\u{ \textrm{\small  The Dyck paths $\tau\big((a_i)_{i=1}^n\big)$}}(0,-2.2)
\Label\u{ \textrm{Figure 5}}(0,-4.5)
\]

\vspace*{3mm}

This procedure works because of the following Lemma whose proof, by induction, is left to the reader.
\begin{lemma}\label{iter}
 \emph{(i)} For $a\in\a_n(021),\ \tau(a)$ ends with $UD$ if and only if $a_n=0$. \\
 \emph{(ii)} For $(a_i)_{i=1}^{n-1}\in\a_{n-1}(021)$, the number of main values for $a_n$ is equal to the number of $DD$-components in  $\tau\big((a_i)_{i=1}^{n-1}\big)$. 
 \end{lemma} 
We leave the reader to verify that the two descriptions give the same bijection. A different bijection from $\a(021)$ to Dyck paths appears in \cite{callan021ascent}.

\section{A \gf for $\mathbf{\a(021)}$ }\label{gf021}
With $\a(021)$ the set of all 021-avoiding ascent sequences,
let $F(x,y,z,w)$ denote the \gf for $\a(021)$ with 
$x,y,z,w$ marking, respectively, length, \# 0's, \# isolated 0's, 
\# runs of 2 or more 0's. For example, 00102200023030 contributes $x^{14} y^8 z^3 w^2$ to $F$.
\begin{theorem}
$
F(x,y,z,w)=
$
\begin{equation}\label{4vargf}
\frac{1}{2x}\left(1 - \sqrt{ \frac{1 - x (4 + y) + 4 x^2 \big(1 + y (1 - z)\big) - 4 x^3 y \big(1 + w y - (1 + y) z\big) +  4 x^4 y^2 (w - z)}{1 -  x y}}\right)
\end{equation}

\end{theorem}
\begin{proof}

A \emph{peak} in a Dyck path is an occurrence of $UD$. 
A \emph{run} of peaks in a Dyck path is a maximal subpath of the form $(UD)^i,\ i\ge 1$ and a 
\emph{good run} of peaks is one that is not immediately followed by a $U$, i.e., is either 
followed by a $D$ or ends the path. A good peak is one that is contained in a good run of peaks.
For example, in the $a_n=1$ path in Figure 5, there are 3 runs of peaks of which the first and third are good.

From the algorithmic description of the bijection $\tau$ from $\a_n(021)$ to $\d_n$ of the previous Section, it is
clear that 0's go to good peaks, indeed the runs of 0's, say of lengths $r_1,r_2,\dots,r_t$ left to right, go to 
the good runs of peaks, also $t$ in number and of lengths $r_1,r_2,\dots,r_t$ left to right.

So our desired \gf $F(x,y,z,w)$ is also a \gf for Dyck paths 
with $x,y,z,w$ marking, respectively, semilength, \# good peaks, \# good peak runs 
of length 1, \#  good peak runs of length $\ge 2$. It is easy to find this \gf: split Dyck paths 
into two classes, (i) sawtooth paths, $(UD)^i,\ i\ge 0$, and (ii) non-sawtooth paths, that is, 
paths of the form $(UD)^iUPDQ$ with $i\ge 0,\ P$ a nonempty Dyck path, and $Q$ a Dyck path. The 
contributions to $F$ are as follows. Class (i) contributes the constant 1 (for the empty path) 
+ $xyz$ (for the path $UD$) + $\sum_{i\ge 2}x^iy^iw$ (for $(UD)^i$ with $i\ge 2$). Class (ii) 
contributes $\sum_{i\ge 0}x^{i+1}(F-1)F$. Consequently,
\[
F = 1+xyz+\frac{x^2y^2w}{1-xy}+\frac{x}{1-x}(F-1)F\, ,
\]
with solution (\ref{4vargf}).
\end{proof}
Several distributions can be derived from $F$. For instance, the number of 0's in 021-avoiding ascent sequences has \gf 
\[
F(x,y,1,1)=\frac{1 - \sqrt{(1 - x (4 + y) + 4 x^2)/(1 - x y) }}{2 x}\,,
\]
sequence \htmladdnormallink{A175136}{http://oeis.org/A175136}.

\section{021-Avoiding ascent sequences -- no 00's} 
Let $\u_n$ denote the set of 021-avoiding ascent sequences of length $n$ that contain no two consecutive zeros and $u_n=\v \,\u_n \v$. For example, $\u_0=\{\epsilon\}$ where $\epsilon$ is the empty sequence, $\u_1=\{0\}$, $\u_2=\{01\}$, and $\u_3=\{010,\,011, 012\}$.
The \gf $\sum_{n\ge 0}u_n x^n$ is given by 
\[
F(x,1,1,0)=\frac{1 - \sqrt{1 - 4x + 4x^3}}{2 x}\,.
\]
We will show bijectively (Theorem \ref{no00} below) that $u_n$ satisfies a Catalan-like recurrence, see
\htmladdnormallink{A025265}{http://oeis.org/A025265} and 
\htmladdnormallink{A025262}{http://oeis.org/A025262}.

We say an entry $a_i$ in $a=(a_i)_{i=1}^n \in \u_n$ is a \emph{max} if $i$ is a tight index. 
We have $a_1=0$ and $a_2=1$ for all $a\in \u_n$ with $n\ge 2$. In particular, $a_1$ is never a max and $a_2$ is always a max---the trivial one. A nontrivial max is one with index $\ge 3$. 

\begin{theorem}\label{no00}
 \begin{equation}\label{catrec}
 u_n=u_0 u_{n-1} +u_1 u_{n-2} + \dots +u_{n-1}u_{0}
 \end{equation}
for $n\ge 3$ with initial conditions $u_0=u_1=u_2=1$.
\end{theorem} 

The following two Propositions (with an intervening lemma) yield that, for $n\ge 3$, 
$u_n=2u_{n-1}+u_{n-2} +\sum_{k=4}^{n}u_{k-3}\,u_{n-k+2}$, which is equivalent to (\ref{catrec}), and Theorem \ref{no00} follows.
\begin{prop} For $n\ge 3,\\
\begin{array}{ll}
(i) & \v\,\{a\in \u_n: a_3=1\}\v =u_{n-1},\\
(ii)& \v\,\{a\in \u_n: a_3=2\}\v =u_{n-1},\\
(iii)& \v\,\{a\in \u_n: a_3=0\textnormal{ and $a$ has no nontrivial 
max}\}\v =u_{n-2}.\\
\end{array}$
\end{prop}
\begin{proof} (i)\ ``Delete $a_3$'' is a bijection from $\{a\in \u_n: a_3=1\}$ to $\u_{n-1}$.\\
(ii)\, Similar to (i), delete $a_3$ and subtract 1 from each later nonzero entry.\\ 
(iii) ``Delete $a_1,\,a_2$'' is a bijection from the $a$'s counted on the left side to $\u_{n-2}$. 
Deleting $a_1=0$ and $a_2=1$ does not introduce a violation of the defining condition for ascent sequences because $a_3,\dots,a_{n}$ are not max entries in $a$.
\end{proof}
\begin{lemma} \label{lastmax} For $n\ge 4$,\\
$\v\,\{a\in \u_n: a_3=0\textnormal{ and $a_n$ is the first (and only) nontrivial max}\}\v =u_{n-3}$\,.
\end{lemma}
\begin{proof}
``Delete $a_1,\,a_2,$ and $a_n$'' is a bijection from the $a$'s counted on the left side to $\u_{n-3}$. 
\end{proof}
\begin{prop}
 For $4\le k \le n,\\  
\v\,\{a\in \u_n: a_3=0\textnormal{ and $a_k$ is the first nontrivial max}\}\v =u_{k-3}\,u_{n-k+2}.$
\end{prop}
\begin{proof}
Suppose $a\in \u_n$ is counted by the left side. According to Lemma \ref{lastmax}, there are $u_{k-3}$ possibilities for the prefix $a_1\dots a_k$. For each such prefix, ``delete $a_1,\dots,a_{k-2}$, change $a_{k-1}$ to 0, and subtract $a_k-1$ from all remaining nonzero entries'' is a bijection to $\u_{n-(k-2)}$.
For example, with $k=6$ and prefix $010124$, this map sends 01012445057 to 0112024.
\end{proof}

\section{021-Avoiding ascent sequences -- no isolated zeros}
Let $\V_n$ denotes the set of 021-avoiding ascent sequences of length $n$ that contain no isolated zeros and $v_n=\v \,\V_n \v$. Thus, $\V_0=\{\epsilon\}$, $\V_1=\{\}$, $\V_2=\{00\}$, and $\V_3=\{000,\,001\}$. 

It is remarkable that $v_n$ satisfies the same recurrence as $u_n$ (``no consecutive 0's'') 
differing only in the initial conditions.
\begin{theorem}\label{noIsolated0}
For $n\ge 1,$
\begin{equation}\label{explicit00}
 v_n=\sum_{k=1}^{\lfloor (n+1)/3 \rfloor}2^{n-3k+1}\binom{n-k-1}{2k-2}\,C_{k-1}
\end{equation}
where $C_n:=\binom{2n}{n}-\binom{2n}{n-1}$ is the Catalan number,   
and $v_n$ satisfies the defining recurrence   
\[ 
 v_n=v_0 v_{n-1} +v_1 v_{n-2} + \dots +v_{n-1}v_{0}
\]
for $n\ge 3$ with initial conditions $v_0=1,\,v_1=0,\,v_2=1$.
\end{theorem}
\begin{proof}
Refine $\V_n$ to 
$\V_{n,k}=\{a\in \V_n:\ a\textrm{ has $k$ runs of 0's}\}$ and set $v_{n,k}=\v \,
\V_{n,k} \v$. For example, $\V_{8,3}=\{00100100,\,00100200\}$. Now let $\b_{n,k}$ 
denote the set of 021-avoiding ascent sequences with $k$ runs of 0's, and 
let $b_{n,k}=\v \,\b_{n,k} \v$. Then ``append 
a 0 to each run of 0's'' is a simple bijection from 
$\b_{n-k,k}$ to $\V_{n,k},\ 1\le k \le (n+1)/3$.

Turning to $\b_{n,k}$, the bijection $\tau$ of Sec. \ref{gf021} sends descents to $DDU$s. 
Since each descent is necessarily to 0 by Lemma \ref{021crit}, 
the number of runs of 0's is precisely $1+\#\,$descents. It is well known that the number of 
Dyck paths of size $n$ with $k\ DDUs$ is given by $t_{n,k}:= 2^{n - 2 k - 1} 
\binom{n - 1}{2 k}C_k$, the Touchard distribution 
\htmladdnormallink{A091894}{http://oeis.org/A091894}. 
Since $b_{n,k}=t_{n,k-1}$ and $v_{n}=\sum_{k\ge 1}v_{n,k}=\sum_{k\ge 1}b_{n-k,k}$, identity (\ref{explicit00}) follows.

The \gf $\sum_{n\ge 0}v_n x^n$ is given by 
\[
F(x,1,0,1)=\frac{1 - \sqrt{1 - 4x + 4x^2 - 4x^3}}{2 x}\, ,
\]
and it is routine to check that the recurrence of the theorem is equivalent to this \gf. 
There does not, however, seem to be any very obvious bijective proof of the recurrence.

\end{proof}

\section{021-Avoiding descent sequences}
Let $ \d_{n,k}$ refer to descent sequences of length $n$ with $k$ descents.

\begin{theorem}\label{021des}
For $n\ge 1,\, k\ge 0,$
\begin{equation}\label{021deseq}
\v \d_{n,k}(021)\v =  \binom{n+k}{3k+1} C_k\,.
\end{equation}
\end{theorem}

\begin{proof}
Consider $w \in \d_{n,k}(021)$.  
By Lemma \ref{021crit}, the descent bottoms of $w$ are all 0, and the descent tops, say $(d_i)_{i=1}^k$, 
form a Catalan sequence, 
that is, $1 \le d_i \le i$ for all $i$ and the $d_i$'s are weakly increasing. 
The number of Catalan sequences 
of length $k$ is well known to be $C_k$, and so there are $C_k$ possibilities for the list of descent tops. 
For each such list $(d_i)_{i=1}^k$, we will show there are $\binom{n+k+1}{3k+1}$ possibilities for $w$.
This is because $w$ must have the  form illustrated below for $k=4$ and $(d_i)_{i=1}^k=\{1,2,2,3\}$, 
\[
0\,0^{*}\,1^{*}\,\boxed{10}\,0^{*}\,1^{*}\,2^{*}\, \boxed{20}\, 0^{*}\,2^{*}\,\boxed{20}\, 0^{*}\,2^{*}\,3^{*}\,\boxed{30}\,0^{*}\,3^{*}\,4^{*}\,5^{*}
\]
where the boxed pairs are the descents and the asterisk in $j^*$ indicates a run of zero or more $j$'s.

First, all 0's must immediately follow the descent bottom 0's (else a new descent is introduced). 
Second, the nonzero entries are weakly increasing (Lemma \ref{021crit}) up to a maximum of $k+1$; 
this accounts for the 
appearance of a single $1^*,2^*,\dots,(k+1)^*$ left to right. But also, if $\boxed{j0}$ is a  
box, then $j^*$ appears both before and after the box, in other words, repetitions account for an additional 
$k$ nonzero asterisks. 
Thus the total number of  asterisks is  $k+1$ (for the 0's) 
+ $k+1$ (for $1,2,\dots,k+1$) + $k$ (for the nonzero repetitions) = $3k+2$.

So, to specify $w$, we have to split $n-(2k+1)$ $x$'s (entries belonging to the asterisks) into $3k+2$ 
segments (some may be empty). By elementary combinatorics, we need to arrange in a row $n-(2k+1)$ $x$'s and 
$3k+1$ ``dividers''---$\binom{n-(2k+1)+3k+1}{3k+1}=\binom{n+k}{3k+1}$ ways.
\end{proof}
It follows routinely from Theorem \ref{021des} that the \gf $G(x,y)$ for nonempty 021-avoiding descent sequences,  
with $x,y$ marking length and number of descents respectively, is given by
\[
G(x,y) = 1+\frac{x}{(1 - x)^2}\, C\left(\frac{x^2 y}{(1 - x)^3}\right) 
= \frac{(1-x) \left(1-\sqrt{1-\frac{4 x^2 y}{(1-x)^3}}\right)}{2 x y}
\]
where $C(x)=\frac{1-\sqrt{1-4x}}{2x}$ is the \gf for the Catalan numbers. 

The \gf for $\d(021)$ by length is thus 
\[
1+G(x,1)= \frac{1+x-\sqrt{\frac{1-3x-x^2-x^3}{1-x}}}{2 x}=1+x+2x^2+4x^3+9x^4+22x^5+57x^6+\cdots\,.
 \]
This is also the \gf for Dyck paths that avoid $UUD\,U$, \htmladdnormallink{A105633}{http://oeis.org/A105633}.
A bijection to explain this equipotence would be interesting.

\vspace*{5mm}
\noindent Department of Statistics, University of Wisconsin-Madison, Madison, WI \ 53706-1532  \\
\noindent{\bf callan@stat.wisc.edu}  \\


\begin{thebibliography}{99}

\bibitem{oeis} The On-Line Encyclopedia of Integer Sequences, published electronically at \htmladdnormallink{http://oeis.org}{https://oeis.org/}, 2019.

\bibitem{dukes2019}  Mark Dukes and Peter R. W. McNamara,  Refining the bijections among ascent sequences, (2+2)-free posets, integer matrices and pattern-avoiding permutations, \emph{J. Combin. Theory Ser. A} \textbf{167} (2019), 403--430. 

\bibitem{dun2011} Paul Duncan and Einar Steingr\'{i}msson, Pattern avoidance in ascent sequences,
Electronic Journal of Combinatorics \textbf{18} (2011), \#\,P226.

\bibitem{021chen} Chen, William Y.C. et al, On 021-Avoiding Ascent Sequences,
\emph{The Electronic Journal of Combinatorics} \textbf{20(1)}  (2013) Art. P76.


\bibitem{callan021ascent} David Callan, Another bijection for 021-avoiding ascent sequences,
arXiv:1402.5898 [math.CO], 2014.



\end{thebibliography}
\end{document}
